\documentclass[12pt,a4paper]{article}
\usepackage[OT1]{fontenc}
\usepackage[latin1]{inputenc}
\usepackage[english]{babel}
\usepackage{amsfonts}
\usepackage{amsthm}

\newtheorem{teo}{Theorem}[section]
\newtheorem{prop}[teo]{Proposition}
\newtheorem{lem}[teo]{Lemma}
\newtheorem{coro}[teo]{Corollary}

\theoremstyle{definition}
\newtheorem{rem}[teo]{Remark}

\def\h{{\cal H}}
\def\b{{\cal B}}
\def\a{{\cal A}}
\def\jj{{\cal J}}
\def\bpe{{\cal B}_p({\cal H})}
\def\upe{U_p({\cal H})}
\def\o{ {\cal O} }
\def\k{ {\cal K} }
\def\u{ {\cal U} }

\begin{document}

\title{\vspace*{0cm}The rectifiable distance in the unitary Fredholm group\footnote{2000 MSC. Primary 22E65;  Secondary 58B20,
58E50.}}
\date{}
\author{Esteban Andruchow and Gabriel Larotonda\footnote{Both authors partially supported by Instituto Argentino de Matemática and CONICET}}

\maketitle

\begin{abstract}
Let $U_c({\cal H})=\{u: u \mbox{ unitary and } u-1 \mbox{ compact}\}$ stand for the unitary Fredholm group. We prove the following convexity result. Denote by $d_\infty$  the rectifiable distance induced by the Finsler metric given by the operator norm in $U_c({\cal H})$. If $u_0,u_1,u\in U_c({\cal H})$ and the geodesic $\beta$ joining $u_0$ and $u_1$ in $U_c({\cal H})$ verifies $d_\infty(u,\beta)<\pi/2$, then the map $f(s)=d_\infty(u,\beta(s))$ is convex for $s\in[0,1]$. In particular the convexity radius of the geodesic balls in $U_c({\cal H})$ is $\pi/4$. The same convexity property holds in the $p$-Schatten unitary groups $U_p({\cal H})=\{u: u \mbox{ unitary and } u-1  \mbox{ in the } p \mbox{-Schatten class}\}$, for $p$ an even integer, $p\ge 4$ (in this case, the distance is strictly convex). The same  results hold in the unitary group of a $C^*$-algebra with a faithful finite trace. We apply this convexity result to establish the existence of curves of minimal length with given initial conditions, in the unitary orbit of an operator, under the action of the Fredholm group. We characterize self-adjoint operators $A$ such that this orbit is a submanifold (of the affine space $A+{\cal K}({\cal H})$, where ${\cal K}({\cal H})$=compact operators).\footnote{Keywords and phrases: finite trace, Fredholm operator, geodesic convexity, homogeneous space, short geodesic, unitary group.}
\end{abstract}

\section{Introduction}
As in the finite dimensional setting, an infinite dimensional manifold with a Finsler metric (a continuous distribution of norms in the tangent spaces) becomes a metric space. The distance between two points is given by the infimum of the lengths of the smooth curves which join these points. The first problem one encounters is the existence of curves of minimal length, or metric geodesics. In the finite dimensional case, compactness of closed balls is a key fact in the solution of this problem. For infinite dimensional manifolds, in the absence of compactness, convexity properties of the distance can prove a useful substitute. In this paper we study manifolds which carry a transitive action of the unitary Fredholm group (defined below), i.e. homogeneous spaces of the Fredholm group.  We start by considering the metric geometry of the group. It is well-known that with the Finsler metric given by the usual operator norm, geodesics (starting at $1$) are one parameter unitary groups $e^{itx}$  with $x^*=x$ compact, and  remain minimal if $|t|\le \frac{\pi}{\|x\|}$.  We prove a convexity property for the metric in the group: if three elements $u_0,u_1,u_2$ of the group lie close enough, the distance from $u_0$ to the points of the geodesic joining $u_1$ and $u_2$, is a convex map in the time parameter. In particular, the radius of convexity of the geodesic balls in the unitary group is $\pi/4$.  With a slightly different proof, we obtain the same result for the unitary group $U_{\cal A}$ of a $C^*$-algebra ${\cal A}$ with a finite faithful trace. 

This results are used to establish the existence of metric geodesics with given initial conditions in certain homogeneous spaces of the Fredholm group. More precisely, in the orbits of operators under the inner action of this group.

Let $\h$ be a Hilbert space. Denote by $\b(\h)$, $\k(\h)$ and $\b_p(\h)$ respectively, the Banach spaces of bounded, compact and $p$-Schatten operators. Let $U(\h)$ be the unitary group of $\h$. This paper is a sequel of \cite{upe}, where we studied the geometry of homogeneous spaces of the classical Lie-Banach  subgroups of $U(\h)$ 
$$
U_p(\h)=\{u\in U(\h): u-1\in\b_p(\h)\}.
$$
The Lie algebra of $U_p(\h)$ is the space $\b_p(\h)_{ah}$ of anti-hermitian operators in $\b_p(\h)$. Thus the natural Finsler metric to consider in $U_p(\h)$ is the $p$-norm, namely $\|x\|_p=Tr(|x|^p)^{1/p}$. We denote by $\|x\|$ the usual supremum norm of $x$. 

In this paper we shall focus on the (unitary) Fredholm group
$$
U_c(\h)=\{u\in U(\h): u-1\in\k(\h)\}.
$$

\bigskip

Among the facts proved in \cite{upe}, let us cite the following, which also hold in the unitary group $U_{\cal A}$ of a $C^*$-algebra ${\cal A}$ with a finite faithful trace \cite{ea}:

\begin{enumerate}
\item
With the Finsler metric given by the $p$-norm, if $p\ge 2$, the curves of the form $\mu(t)=ue^{tx}$ ($u\in U_p(\h)$, $x\in\b_p(\h)_{ah}$, $\|x\|\le \pi$) have minimal length (and are unique with this property if $\|x\|<\pi$) as long as $|t|\le 1$. Any pair $u_0,u_1$ of elements in $U_p(\h)$ can be joined by a minimal geodesic, which is unique if $\|u_0-u_1\| <2$.
\item
If $p$ is an even integer, $d_p$ denotes the distance given by the Finsler metric, $u_0\in U_p(\h)$ and $\delta$ is a geodesic such that $d_p(u_0,\delta(t))<\pi/4$ for $t$ in a certain interval $I$, then the 
function $f(t)=d_p(u_0,\delta(t))^p$ is strictly convex for $t\in I$.
\end{enumerate}

\medskip

Our first result establishes that the second cited result can be refined: the function 
$$
g(t)=d_p(u_0,\delta(t))
$$
is strictly convex in $I$ for $p$ even, $p\ge 4$. As a consequence, if $u,v\in U_c(\h)$, and $x=-x^*\in\k(\h)$, with $\|x\|\le \pi/4$, the function
$$
g_\infty(t)=d_\infty(u, ve^{tx}), \ \ t\in[0,1]
$$
is convex. 

In order to obtain this result, we need the analogous to the first cited result to hold for the group $U_c(\h)$. This is perhaps well-known, we sketch a proof at the beginning of the next section. With the same proof, the result holds for $C^*$-algebras with a finite faithful trace.

\medskip

We apply the convexity property of the metric to the geometric study of the unitary orbits of an  operator $A$ under the Fredholm group,
$$
\o_A=\{uAu^*: u\in U_c(\h)\}.
$$
We deal more extensively with the case $A^*=A$, though we consider an example of a non self-adjoint operator.
This orbit lies inside the affine space $A+\k(\h)$, and therefore its tangent spaces lie inside $\k(\h)$. A natural Finsler metric to consider here, following the ideas in \cite{duranmatarecht}, is the quotient metric induced by the usual operator norm. We prove first that if $A^*=A$, then $\o_A$ is a complemented submanifold of $A+\k(\h)$ and a smooth homogeneous space if and only if the spectrum of $A$ is finite. Then we proceed on the study of the existence of minimal curves in $\o_A$ with given initial conditions. If $x=-x^*\in \k(\h)$ and $b\in\o_A$,  $xb-bx$ is a typical tangent vector of $\o_A$ at $b$. In the case when  $A$ has finite rank (or more generally, $A$ has finite spectrum with all but one of the eigenspaces of finite dimension) we prove that there exists a compact anti-hermitian operator $z_c$ such that
\begin{enumerate}
\item
$xb-bx=z_cb-bz_c$, and
\item
$\|z_c\|\le \|x'\|$ for all $x'\in\k(\h)_{ah}$ such that $x'b-bx'=xb-bx$,
\end{enumerate}
and the curve $\epsilon(t)= e^{tz_c}be^{-tz_c}$  has minimal length up to $|t|\le \frac{\pi}{4\|z_c\|}$. Such elements $z_c$  are called in \cite{duranmatarecht} {\it minimal liftings} of the tangent vector $xb-bx$ because they achieve the quotient norm, they need not be unique. They are found using a matrix completion technique developed by C. Davis, W. M. Kahan and H. F. Weinberger \cite{davis}, in their solution to M. G. Krein's extension problem. These minimal liftings provide, in a different but related framework \cite{duranmatarecht}, curves of minimal length in homogeneous spaces of the unitary group of a von Neumann algebra. 

\bigskip

The contents of the paper are as follows. In Section 2 we prove the convexity results for the geodesic distance in $U_p(\h)$ ($p$ even, $p\ge 4$) and $U_c(\h)$. In Section 3 we show that this result holds in $C^*$-algebras with a finite and faithful trace. In Section 4 we characterize self-adjoint operators $A$ such that $\o_A$ is a submanifold of $A+\k(\h)$. In Section 5 we introduce a Finsler metric in $\o_A$, and show that if, for instance, $A$ has finite rank, then one can find curves of minimal length (or metric geodesics) in $\o_A$ with given initial conditions (position and velocity). In Section 6 we further specialize to the case $A=P$ an orthogonal projection, and show that in $\o_P$ one can find metric geodesics joining any given pair of points. In Section 7 we consider the example of a non self-adjoint operator, namely an order two nilpotent $N$. We show that for certain special directions, called here anti-symmetric tangent vectors, metric geodesics with these special vectors as velocity vectors exist.

\section{Convexity in the classical unitary groups}

In this paper, we consider piecewise smooth curves $\gamma:[a,b]\to X$ taking values in either a unitary group, or an homogeneous space of a unitary group. By smooth we mean $C^1$ and with non vanishing derivative. The rectifiable length of $\gamma$ is defined in the usual way,
$$
L(\gamma)=\int_a^b \|\dot{\gamma}\|_{\gamma},
$$
where $\|v\|_x$ denotes the norm of a tangent vector $v$ at $x\in X$. The rectifiable distance is defined as the infimum of the lengths of smooth curves joining its given endpoints, that is
$$
d(x,y)=\inf\{ L(\gamma): \,\gamma\mbox{ is smooth}, \, \gamma(a)=x,\gamma(b)=y\}.
$$
For $p\ge 1$, we use $d_p$ to indicate the rectifiable distance in the group $U_p(\cal H)$, and $d_{\infty}$ to indicate the rectifiable distance in either the Fredholm group or a given $C^*$-algebra (i.e. the distance obtained by measuring curves with the uniform norm).

\begin{lem}\label{corta}
Let $x\in\k(\h)_{ah}$ such that $\|x\|\le \pi$, and $u\in U_c(\h)$. Then the curve $\epsilon(t)=ue^{tx}$ has minimal length for $t\in[0,1]$. Any pair $u,v\in U_c(\h)$ can be joined by such a curve.
\end{lem}
\begin{proof}
Clearly we can suppose $u=1$. Note that $x^2\le 0$, and since it is compact, there exists a unit vector $\xi\in\h$, which is a norming eigenvector for $x^2$, i.e. $x^2\xi=-\|x\|^2\xi$. Consider the map
$$
\rho: U_c(\h)\to S_\h=\{\eta\in\h: \|\eta\|=1\} , \ \ \rho(u)=u\xi.
$$
Note that $e(t) =\rho(\epsilon(t))$ is a geodesic of the sphere $S_\h$. Indeed, if $k=\|x\|$
$$
\ddot{e}(t)=e^{tx}x^2\xi=-k^2e^{tx}\xi=-k^2e(t).
$$
Moreover, since $k\le \pi$, it is a minimal geodesic of the sphere for $t\in[0,1]$. Let $\gamma(t)$, $t\in[0,1]$ be a curve in $U_c(\h)$ joining the same endpoints as $\epsilon$. Then $\rho(\gamma)$ joins the same endpoints as $e$. Thus $L(e)\le L(\rho(\gamma))$. On the other hand
$$
L(\rho(\gamma))=\int_0^1\|\dot{\gamma}(t)\xi\| d t \le \int_0^1 \|\dot{\gamma}(t)\| d t =L(\gamma).
$$
Moreover, 
$$
\|\dot{e}\|^2=\|x\xi\|^2=\langle x\xi,x\xi \rangle =-\langle x^2\xi,\xi\rangle =\|x^2\|=\|x\|^2,
$$
and thus $L(e)=L(\epsilon)$. This completes the first part of the proof. 

Given $w\in U_c(\h)$, there exists $x\in\k(\h)_{ah}$ such that $w=e^x$ \cite{harpe}. Moreover, $x$ can be chosen with $\|x\|\le \pi$. 
\end{proof}

\begin{rem}\label{jesiano}
The Hessian of the $p$-norms was studied in \cite{convexg,cocomata}. We recall a few facts we will use in the proof of the next theorem. Let $a,b,c\in \bpe_{ah}$, let $H_a:\bpe_{ah}\times \bpe_{ah}\to \mathbb R$ stand for the symmetric bilinear form given by
$$
H_a(b,c)=(-1)^{\frac{p}{2}}p \sum_{k=0}^{p-2}\, Tr(a^{p-2-k}b a^k c).
$$
Denote by $Q_a$ the quadratic form induced by $H_a$. Then (cf. Lemma 4.1 in \cite{convexg} and equation (3.1) in \cite{cocomata}):
\begin{enumerate}
\item $Q_a([b,a])\le 4 \|a\|^2 Q_a(b)$.
\item $Q_a(b)=p\|ba^{\frac{p}{2}-1}\|_2^2+\frac{p}{2}\sum_{l+m=n-2}\|a^l (ab+ba)a^m\|_2^2$.
\end{enumerate}
In particular $H_a$ is positive definite for any $a\in \bpe_{ah}$. These facts were proved in the context of a C$^*$-algebra with a finite trace. They hold in our context here, because they rely on the abstract properties of the trace, and on the following facts, which are elementary:
$$
\|xy\|_p\le \|x\| \|y\|_p, \ \ \| |x|y\|_p=\|xy\|_p , \ \ x\in \b(\h), y\in\bpe.
$$
\end{rem}

Our convexity results follow. If $u\in\upe$, denote by 
$$
B_p(u,r)=\{w\in\upe: d_p(u,w)<r\}
$$
the ball of radius $r$ around $u$ in $\bpe$. First we prove a refinement of a result from \cite[Theorem 3.6]{upe}. Let $g(r)={\rm sinc}(r)=r^{-1}\sin(r)$, which is a strictly decreasing function in $[0,\pi]$, with $g(0)=1$ and $g(\pi)=0$.

\medskip

\begin{teo}\label{convexidad_p}
Let $p$ be an even integer, $p\ge 4$. Let $u\in \upe$ and  let $\beta:[0,1]\to\upe$ be a non constant geodesic contained in the geodesic ball of radius $r_p=\frac12 g^{-1}(\frac{1}{p-1})$, namely $\beta\subset B_p(u,r_p)$. Assume further that $u$ does not belong to any prolongation of $\beta$. Then  $f_p(s)=d_p(u,\beta(s))^p$ verifies 
$$
(p-1)f'_p(s)^2\le pf_p(s) f_p''(s) \frac{1}{g(2\|w_s\|_p)}\le p(p-1)f_p(s)f''_p(s),
$$
for any $s\in [0,1]$, and equality holds in the last term if and only if $f_p''(s)=0$.
\end{teo}
\begin{proof}
We may assume that $u=1$ since the action of unitary elements is isometric. Let
$v,z\in \bpe_{ah}$ such that $\beta(s)=e^ve^{sz}$. Since
$$
\|\beta(s)-1\|\le \|\beta(s)-1\|_p\le d_p(\beta(s),1)<r_p<1,
$$
the curve $\beta$ has analytic logarithm given by the smooth curve of (anti-hermitian) operators $w_s=\log(\beta(s))\in \bpe_{ah}$. Let $\gamma_s(t)=e^{tw_s}$. Since $\|w_s\|\le \|w_s\|_p<r_p<\pi/4$ ($p\ge 4$), then $\gamma_s$ is a short geodesic joining $1$ and $\beta(s)$, of length
$\|w_s\|_p=d_p(1,\beta(s))$. Then $f_p(s)=\|w_s\|_p^p=Tr((-w_s^2)^\frac{p}{2})=(-1)^{\frac{p}{2}}Tr(w_s^p)$, hence 
$$
f'_p(s)=
(-1)^{\frac{p}{2}}p\, Tr(w_s^{p-1} \dot{w_s})=\frac{1}{p-1} H_{w_s}(\dot{w_s},w_s).
$$
For $x,y\in \bpe_{ah}$, we have the well-known formula $d\,
\exp_x(y)=\int_0^1 e^{(1-t)x}ye^{tx}\,dt$ (cf. \cite[Lemma 3.3]{upe}). Since 
$e^{w_s}=e^ve^{sz}$, then $e^{-w_s}\; d\; \exp_{w_s} (\dot{w_s}) =z$, namely
\begin{equation}\label{difexp}
z= \int_0^1 e^{-tw_s} \dot{w_s} e^{tw_s}\;dt.
\end{equation}
Thus $Tr(w_s^{p-1} \dot{w_s} )=\int_0^1 Tr(w_s^{p-1}e^{-tw_s} \dot{w_s} e^{tw_s})\;dt=Tr(zw_s^{p-1})$.
Hence 
$$
f''_p(s)=
(-1)^{\frac{p}{2}}p \sum_{k=0}^{p-2}\, Tr(w_s^{p-2-k}\dot{w_s}w_s^k z)=H_{w_s}(\dot{w_s},z),
$$
and again by equation (\ref{difexp}) above,
$$
f_p''(s)=\int_0^1 H_{w_s}(\delta_s(0),\delta_s(t))dt
$$
where $\delta_s(t)=e^{-tw_s}\dot{w_s} e^{tw_s}$. Suppose that for this value of $s\in[0,1]$, $R_s^2:=Q_{w_s}(\dot{w_s})\ne 0$, where $Q_{w_s}$ is the quadratic form associated to $H_{w_s}$. Then, if $K_s\subset \bpe_{ah}$ is the null space of $H_{w_s}$, consider the quotient space $\bpe_{ah}/K_s$ equipped with the inner product $H_{w_s}(\cdot,\cdot)$. An elementary computation shows that $\delta_s$ lives in a sphere of radius $R_s$ of this pre-Hilbert space, hence
$$
H_w(\delta_s(0),\delta_s(t))=R_s^2 \cos\alpha_s(t),
$$
where $\alpha_s(t)$ is the angle subtended by $\delta_s(0)$ and $\delta_s(t)$.

Note that
$$
\frac{d}{d t} \delta_s(t)=e^{-tw_s}[w_s,\dot{w}_s]e^{tw_s},
$$
and that $Q_s$ is $Ad_{e^{-tw_s}}$-invariant:
$$
Q_s(e^{-tw_s}ae^{tw_s})=Q_s(a).
$$
Then, reasoning in the sphere,
\begin{eqnarray}
R_s\alpha_s(t)&\le& L_0^t(\delta_s)=\int_0^t Q_{w_s}^{\frac12}(e^{-tw_s}[w_s,\dot{w_s}]e^{t w_s})\,dt\nonumber\\
& =&\int_0^t Q_{w_s}^{\frac12}([w_s,\dot{w_s}])\,dt=t\,Q_{w_s}^{\frac12}([w_s,\dot{w_s}]).\nonumber
\end{eqnarray}
By property $1.$ of Remark \ref{jesiano},
$$
R_s\alpha_s(t)\le t\, 2\|w_s\| R_s \le t 2\|w_s\|_p R_s.
$$
for any $t\in [0,1]$. Hence, since $\|w_s\|_p<\frac{\pi}{2}$, then
$$
\cos(\alpha_s(t))\ge \cos(2t\|w_s\|_p),
$$
and we obtain
$$
f''_p(s)\ge R_s^2\frac{\sin(2\|w_s\|_p)}{2\|w_s\|_p}>0
$$
integrating the expression for $f_p''(s)$ above. Note that the Cauchy-Schwarz inequality for $H_{w_s}$ implies that
$$
(p-1)^2f'_p(s)^2=H_{w_s}^2(w_s,\dot{w_s})\le Q_{w_s}(w_s) Q_{w_s}(\dot{w_s})=p(p-1)f_p(s) R_s^ 2,
$$
and therefore
$$
(p-1)f'_p(s)^2 \le pf_p(s) f_p''(s) \frac{1}{g(2\|w_s\|_p)}.
$$
For the chosen value of $r_p=\frac12 g^{-1}(\frac{1}{p-1})$, one has that 
$$
\frac{1}{g(2\|w_s\|_p)}<p-1,
$$
and therefore the right hand inequality of this theorem follows. Note also that since we are supposing $R_s\ne 0$, clearly $f_p''(s)\ne 0$, and thus for such $s$ the second inequality is strict. 

Suppose now that $R_s=Q_{w_s}(\dot{w}_s)=0$ Then, again by the Cauchy-Schwarz inequality for the form $H_{w_s}$,
$$
0\le f_p''(s)=H_{w_s}(\dot{w}_s, z)\le Q_s(\dot{w}_s)^{1/2}Q_s(z)^{1/2}=0,
$$
and similarly
$$
|f_p'(s)|=\frac{1}{p-1}|H_{w_s}(\dot{w}_s, w_s)|\le \frac{1}{p-1}Q_s(\dot{w}_s)^{1/2}Q_s(w_s)^{1/2}=0,
$$
which concludes the proof.
\end{proof}

The following elementary lemma will simplify the proof of the next corollaries.

\begin{lem}\label{fseg}
Let $C,\varepsilon >0$, let $f:(-\varepsilon,1+\varepsilon)\to \mathbb R$ be a non constant real analytic function such that $f'(s)^2\le C f''(s)$ for any $s\in [0,1]$. Then $f$ is strictly convex in $(0,1)$ and moreover there is at most one point $\alpha\in (0,1)$ where $f''_p(\alpha)=0$.
\end{lem}
\begin{proof}
By the mean value theorem, the condition on $f$ implies that for each pair of roots of $f'$, there is another root in between. Since $f$ is analytic and non constant, $C_0(f')$ is an empty set or has one point $\alpha\in (-\varepsilon,1+\varepsilon)$. If $C_0(f')$ does not meet $(0,1)$, then $f''>0$ there and we are done. We assume then that there exists $\alpha$ in $(0,1)$ such that $f'(\alpha)=0$. Note that  $-f'(x)=f'(\alpha)-f'(x)=\int_x^{\alpha}f''(s)ds> 0$ for any $x\in (-\varepsilon,\alpha]$ and $f'(y)=f'(y)-f_p'(\alpha)=\int_{\alpha}^y f''(s)ds>0$ for any $y\in [\alpha,1+\varepsilon)$, hence $f'$ is strictly negative in $(-\varepsilon,\alpha)$ and strictly positive in $(\alpha,1+\varepsilon)$, so $f$ is strictly convex in each interval. If $f(\alpha)<[f(1)-f(0)]\alpha+f(0)$, we are done. If not, by the mean value theorem there exists $x\in (0,\alpha)$, $y\in (\alpha,1)$ such that
$$
f(1)-f(0)=\frac{f(\alpha)-f(0)}{\alpha}=f'(x)<0
$$
and 
$$
f(1)-f(0)=\frac{f(1)-f(\alpha)}{1-\alpha}=f'(y)>0,
$$
a contradiction.
\end{proof}

\begin{coro}
If $u,\beta$ are as in Theorem \ref{convexidad_p}, then $f_p$ is strictly convex.
\end{coro}
\begin{proof}
Since $f_p$ is analytic and $f_p'(s)^2\le C f''(s)$, if $f_p$ is non constant, then $f_p$ is strictly convex  by Lemma \ref{fseg}. Assume then that $f_p$ is constant with $f_p(s)=f_p(0)=\|v\|_p$ for any $s\in [0,1]$. Note that $R_s\equiv 0$, and then by property $2.$ of Remark \ref{jesiano}, $w_s^{\frac{p}{2}-1}z=0$ and an elementary computation involving the functional calculus of anti-hermitian operators shows that $w_s z=0$. In particular $vz=0$ which implies $w_s=v+sz$ by the Baker-Campbell-Hausdorff formula. But since the norm of $\bpe$ is strictly convex, $w_s$ cannot have constant norm unless $v$ is a multiple of $z$, and in that case, $u$ and $\beta$ are aligned contradicting the assumption of the theorem.
\end{proof}

\begin{coro}
Let $u,\beta$ be as in Theorem \ref{convexidad_p}. Then $g_p(s)=d_p(u,\beta(s))=f_p(s)^{\frac1p}$ is strictly convex.
\end{coro}
\begin{proof}
Computing $g''_p(s)$, we obtain
$$
g_p''(s)=\frac{1}{p^2}f^{\frac1p -2}\left[pf_p(s)f_p''(s)-(p-1)f_p'(s)^2 \right]
$$
which is positive by the previous theorem. Moreover it is strictly positive if and only if $f''_p(s)\ne 0$. By Lemma \ref{fseg}, there is at most one point $\alpha\in (0,1)$ where $f''_p(s)=0$, and arguing as in the proof of that lemma, it follows that $g_p$ is strictly convex.
\end{proof}

\begin{rem}
Note that for $p=2$  the above statements are incomplete: $f_2$ is certainly strictly convex if $\beta$ and $u$ are as above, but it is not clear whether $g_2$ is convex. This last assertion is equivalent to
$$
|Tr(w_sz)|\le |Tr(\dot{w_s}z)|^{\frac12}\|w_s\|_2 =\mu_s \|w_s\|_2 \|\dot{w_s}\|_2
$$
where $\mu_s=\int_0^1 \cos\alpha_s(t)dt\le 1$ is a positive constant, so it is easy to check that $g_2$ is convex if $v$ and $z$ commute.
\end{rem}

Our main result on the convexity of the rectifiable distance in the Fredholm group follows.

\begin{teo}
Let $u\in U_c({\cal H})$, $\beta:[0,1]\to {\cal U}_c({\cal H})$ a geodesic such that $d_\infty(u,\beta)<\frac{\pi}{2}$. Then for $s\in[0,1]$, the function $g(s)=d_\infty(u,\beta(s))$ is convex.
\end{teo}
\begin{proof}
We assume that $u=1$, $\beta(s)=e^ve^{sz}$ as before; here $v$ and $z$ are compact. Since $d_\infty(u,\beta(s))<\frac{\pi}{2}$, then $w_s=\log(\beta(s))$ is a compact anti-hermitian operator, with $\|w_s\|=d(1,\beta(s))=g(s)$. Let $q_n$ be an increasing sequence of finite rank projections which converge strongly to $1$, and consider $v_n=q_nvq_n$ and $z_n=q_nzq_n$. Note that since $v,z$ are compact, $v_n\to v$ and $z_n\to z$ in norm, and therefore $e^{v_n}e^{sz_n}\to e^ve^{sz}$ in norm for all $s\in[0,1]$. Therefore there exists $n_0$ such that for all $n\ge n_0$, $\|e^{v_n}e^{sz_n}-1\|<\pi/4$ for all $s\in[0,1]$. Let $w_{n,s}=\log(e^{v_n}e^{sz_n})$. It follows that the map
$$
g_{n,p}(s)=d_p(1,e^{v_n}e^{sz_n})=\|w_{n,s}\|_p,
$$ 
is convex in $[0,1]$. Note that $w_{n,s}$ is a finite rank operator, thus 
$$
\lim_{p\to \infty}\|w_{n,s}\|_p=\|w_{n,s}\|=:g_{n,\infty}(s).
$$
Thus $g_{n,\infty}$ is convex.
Finally, by continuity of the functional calculus,
$$
g(s)=\|w_s\|=\lim_{n\to \infty}\|w_{n,s}\|=\lim_{n\to\infty} g_{n,\infty}(s)
$$
is convex for $s\in[0,1]$.
\end{proof}

\begin{coro}
The radius of convexity of the geodesic balls in $U_c({\cal H})$ is $\frac{\pi}{4}$.
\end{coro}

\section{Convexity in the unitary group of a C$^*$-algebra with a finite trace}

Several results in the previous section remain valid when properly rephrased in the context of a finite C$^*$-algebra $\a$ with a faithful and finite trace $\tau$.  To begin with, Remark \ref{jesiano} was originally stated in this context \cite{cocomata}. Denote by $U_\a$ the unitary group of $\a$. Then the one parameter groups with skew-adjoint exponent $x$, with $\|x\|<\pi$ are the short curves in $U_{\a}$, when we measure them with the $p$-norms induced by the trace (see \cite[Theorem 5.4]{ea} for a proof).  That the same assertion holds if we measure curves with the uniform norm of $\a$ is perhaps well-known, we include a proof of it.

\begin{lem}
Let $x\in \a$ be skew-adjoint, with $\|x\|<\pi$, and $u\in U_\a$. Then the curve $\epsilon(t)=ue^{tx}$ has minimal length for $t\in[0,1]$. Any pair $u,v\in U_\a$ such that $\|u-v\|<2$ can be joined by such a curve.
\end{lem}
\begin{proof}
Let $\varphi$ be a state of $\a$ such that $\varphi(x)=-\|x\|^2$. Let $H_{\varphi}$ be the Hilbert space associated to the Gelfand-Neimark-Segal representation of $\a$ induced by $\varphi$. Consider the unit sphere of the Hilbert space $H_{\varphi}$, and represent elements $u\in U_\a$ as elements in the sphere by means of a cyclic and separating vector $\xi_{\varphi}\in H_{\varphi}$,
$$
\pi_{\varphi}(u)\xi_{\varphi}.
$$
The argument carries on as in Lemma \ref{corta}, because $e^{tx}$ is mapped to $e(t)=e^{t\pi_{\varphi}(x)}\xi_{\varphi}$, which is a minimal geodesic in the sphere since
$$
\ddot{e}(t)=e^{t\pi_{\varphi}(x)}\pi_{\varphi}(x^2) \xi_{\varphi}=-e^{t\pi_{\varphi}(x)}\|x\|^2 \xi_{\varphi}=-\|x\|^2 e(t).
$$
See \cite{pr} for a similar argument. If $\|u-v\|<2$ then there exists a skew-adjoint element $x\in \a$ with $\|x\|<\pi$ such that $v=ue^x$, and one can consider the curve $\epsilon(t)=ue^{tx}$.
\end{proof}

Let us state now the analogous of Theorem \ref{convexidad_p}. 
\begin{teo}
Let $p$ be a positive even integer, $p\ge 4$. Let $u\in U_\a$ and  let $\beta:[0,1]\to U_\a$ be a non constant geodesic contained in the geodesic (uniform) ball of radius $r_p=\frac12 g^{-1}(\frac{1}{p-1})$, namely $\beta\subset B_\infty(u,r_p)$. Assume further that $u$ does not belong to any prolongation of $\beta$. Then  $f_p(s)=d_p(u,\beta(s))^p$ verifies 
$$
(p-1)f'_p(s)^2\le pf_p(s) f_p''(s) \frac{1}{g(2\|w_s\|)}\le p(p-1)f_p(s)f''_p(s),
$$
for any $s\in [0,1]$, and equality holds in the last term if and only if $f_p''(s)=0$.
\end{teo}
\begin{proof}
One only needs to review the proof of Theorem \ref{convexidad_p}, and note that the result follows using the inequalities with the operator norm $\| \ \|$. Notice that by Remark 3.7 in \cite{upe}, the condition $\|w_s\|<\pi/2$ guarantees that $f_p$ is convex.
\end{proof}

Analogously the other results follow:

\begin{coro}
Let $u$, $\beta$ be as in the above theorem. Then $g_p(s)=d_p(u,\beta(s))=f_p(s)^{\frac1p}$ is strictly convex.
\end{coro}

\begin{coro}
Let $u\in U_\a$, $\beta:[0,1]\to U_\a$ a geodesic such that $d_\infty(u,\beta)<\frac{\pi}{2}$. Then $g(s)=d_\infty(u,\beta(s))$, $s\in[0,1]$ is a convex function.
\end{coro}

\begin{coro}
The radius of convexity of the geodesic balls in $U_\a$ is $\frac{\pi}{4}$.
\end{coro}

\section{Smooth structure of unitary orbits of the Fredholm group}

In this section we shall investigate the geometric structure of the unitary orbit of a fixed operator. To simplify our exposition, we shall consider the case of a bounded self-adjoint operator $A\in \b(\h)$, and its orbit
$$
\o_A=\{uAu^*: u\in U_c(\h)\}.
$$
Note that $\o_A\subset A+\k(\h)_h$ (and clearly $\o_A\subset \k(\h)_h$ if $A$ is compact).
First we characterize the operators $A$ such that $\o_A$ is a submanifold of the affine Banach space $A+\k(\h)$.
The argument is analogous to the one developed in \cite{upe} for the orbit of the group $U_2(\h)$. We shall use the following lemma, which is an elementary consequence of the Inverse Function Theorem. Its proof can be found in the Appendix of \cite{raeburn}.
\begin{lem}\label{lemaraeburn}
Let $G$ be a Banach-Lie group acting smoothly on a Banach space $X$. For a fixed
$x\in X$, denote by $\pi_{x}:G\to X$ the smooth map $\pi_{x}(g)=g\cdot
x$. Suppose that
\begin{enumerate}
\item
$\pi_{x}$ is an open mapping, when regarded as a map from $G$ onto the orbit
$\{g\cdot x: g\in G\}$ of $x$ (with the relative topology of $X$).
\item
The differential $d(\pi_{x})_1:(TG)_1\to X$ splits: its kernel and range are
closed complemented subspaces.
\end{enumerate}
Then the orbit $\{g\cdot x: g\in G\}$ is a smooth submanifold of  $X$, and the
map
$\pi_{x}:G\to \{g\cdot x: g\in G\}$ is a smooth submersion.
\end{lem}
In our case, we must consider the map
$$
\pi_A:U_c(\h)\to \o_A \subset A+\k(\h), \ \ \pi_A(u)=uAu^*,
$$
whose differential at  $1$ is the inner derivation
$$
\delta_A:\k(\h)_{ah}\to \k(\h)_h , \ \ \delta_A(x)=xA-Ax.
$$
There are many results on the spectral properties of elementary operators, in particular of inner derivations. One can find an extensive survey on this subject in the book \cite{afhv} by C. Apostol, L. Fialkow, D.A. Herrero and D.V. Voiculescu. Specifically related to our context, where the inner derivation is restricted to the compact operators is the paper \cite{fialkow} by L. Fialkow. We transcribe a result from this paper: denote by $\tau_{AB}$ the operator $\tau_{AB}(x)=Ax-xB$. Let ${\cal J}$ be any Schatten ideal.
\begin{teo}\label{teofialkow} {\rm (Fialkow \cite{fialkow})}
The following are equivalent:
\begin{enumerate}
\item
$\tau_{AB}:\b(\h)\to \b(\h)$ is bounded below.
\item
$\tau_{AB}:{\cal J}\to {\cal J}$ is bounded below for some ${\cal J}$.
\item
$\tau_{AB}:{\cal J}\to {\cal J}$ is bounded below for any ${\cal J}$.
\item
$\sigma_l(A)\cap \sigma_r(B)=\emptyset$.
\end{enumerate}
\end{teo}
Here $\sigma_l(A)$ (resp. $\sigma_r(B)$) denote the left (resp. right) spectrum of $A$ (resp. $B$).
In our particular case, we deal with $\tau_{AA}=\delta_A$ and ${\cal J}=\k(\h)$. In view of the Lemma above, we must show that $\delta_A$ splits: i.e. that  its range and kernel are closed and complemented.
\begin{teo}
The map $\delta_A:\k_{ah}(\h)\to \k(\h)_h$ has closed range if and only if the spectrum of $A$ is finite. In this case, the map $\delta_A$ splits.
\end{teo}
\begin{proof}
Denote  by $\delta^{\mathbb{C}}_A$ the map $\delta^{\mathbb{C}}_A:\k(\h)\to \k(\h)$ defined accordingly. Clearly 
$$
\k(\h)=\k(\h)_h\oplus \k(\h)_{ah},
$$
$$
\delta^{\mathbb{C}}_A(\k(\h)_h)\subset \k(\h)_{ah} \ \hbox{ and } \  \delta^{\mathbb{C}}_A(\k(\h)_{ah})\subset \k(\h)_{h}.
$$
Therefore it is apparent that $\delta_A$ has closed range (resp. splits) if and only if $\delta^{\mathbb{C}}_A$ does.  Let us denote this latter map also by $\delta_A$ to lighten the notation.

Suppose first that $\delta_A$ has closed range. The Hilbert space $\h$ can be decomposed in two orthogonal subspaces $\h=\h_c\oplus\h_{pp}$ which reduce $A$, such that $A_c=A|_{\h_c}\in \b(\h_c)$ has continuous spectrum, and the spectrum of $A_{pp}=A|_{\h_{pp}}$ has a dense subset of eigenvectors. We claim that  $\delta_{A_c}$ and $\delta_{A_{pp}}$ have both closed range. Suppose $x_n\in\b(\h_c)$ is such that $\delta_{A_c}(x_n)\to y$, then $y_n=x_n\oplus 0 \in \b(\h)$ satisfies
$$
\delta_A(y_n)=\delta_{A_c}(x_n)\oplus 0\to y\oplus 0
$$
in $\b(\h)$, and thus $y\oplus 0=\delta_A(x)$. If one writes this equality in matrix form (in terms of the decomposition $\h=\h_c\oplus\h_{pp}$), one has
$$
\left( \begin{array}{ll} y & 0 \\ 0 & 0 \end{array} \right) = 
\left( \begin{array}{ll} x_{11}A_c-A_cx_{11} & x_{12}A_{pp}-A_cx_{12} \\ x_{21}A_c-A_{pp}x_{21} & x_{22}A_{pp}-A_{pp}x_{22} \end{array} \right),
$$ 
and therefore $y=\delta_{A_c}(x_{11})$. Analogously one proves that the range of $\delta_{A_{pp}}$ is closed. In order to prove our claim (that the spectrum of $A$ is finite), we must show first that $\h_c$ is trivial. First note that $\delta_{A_c}:\k(\h_c)\to \k(\h_c)$ has trivial kernel. Indeed, if $x\ne 0$ is a compact operator commuting with $A_c$, since also $x+x^*$ commutes with $A_c$, by the spectral decomposition of compact self-adjoint operators, one can find a non trivial (finite rank) spectral projection of $x+x^*$, which commutes with $A_c$, and thus $A_c$ would have an eigenvalue, leading to a contradiction. It follows that $\delta_{A_c}:\k(\h_c)\to \k(\h_c)$ is bounded from below. Thus, by Fialkow's theorem above, one would have that $\sigma_l(A_c)\cap\sigma_r(A_c)=\emptyset$. Since for self-adjoint operators, right and left spectra coincide, this implies that the spectrum of $A_c$ is empty, and therefore $\h_c$ is trivial. 

It follows that the spectrum of $A$ has a dense subset of eigenvalues. Suppose that there are infinitely many eigenvalues. By adding a multiple of the identity to $A$ (a change that does not affect $\delta_A$), we may suppose that $0$ is a accumulation point of the set of eigenvalues of $A$. From this infinite set one can select a sequence of (different) eigenvalues $\{\lambda_n: n\ge 1\}$ which are square summable. For each $n\ge 1$ pick a unit eigenvector $e_n$,  consider $\h_0$ the closed linear span of these eigenvectors, and denote $A_0=A|_{\h_0}\in \b(\h_0)$. Note that $A_0$ is a Hilbert-Schmidt operator. It is apparent that since $\delta_A:\k(\h)\to \k(\h)$ has closed range, then  $\delta_{A_0}:\k(\h_0)\to \k(\h_0)$ also has closed range. Let us show that the kernel is complemented. Note that $A_0$, written in the orthogonal basis $\{e_n: n\ge 1\}$, is a diagonal infinite matrix, with different entries in the diagonal.  Thus the kernel of $\delta_{A_0}$, which is formed by the compact operators in $\h_0$ which commute with $A_0$, consists also of diagonal matrices. Therefore $\ker \delta_{A_0}$ is complemented, and one can choose the projection $P$ onto  $\ker \delta_{A_0}$ given by
$$
P: \left( \begin{array}{llll} x_{11} & x_{12} &  x_{13} & \dots \\ x_{21} & x_{22} &  x_{23} & \dots \\
x_{31} & x_{32} & x_{33} & \dots \\ \dots &  &  &  \end{array}\right) \mapsto \left( \begin{array}{llll}  x_{11} & 0 & 0 & \dots \\ 0 & x_{22} & 0 & \dots \\ 0  & 0  & x_{33} & \dots \\ \dots &  &  & \end{array}\right) .
$$
Then $\k(\h_0)=\ker \delta_{A_0}\oplus L$, with $L=\ker P$, and
$$
\delta_{A_0}|_L:L\to R(\delta_{A_0})
$$
is an isomorphism between Banach spaces. It follows that there exists a constant $C>0$ such that
$$
\|xA_0-A_0x\|\ge C\|x-P(x)\|, \ \ \hbox{ for all } x\in \k(\h_0).
$$
For each $k\ge 1$, consider the $k\times k$ matrix $b_k$ with $\frac1k$ in all entries, and let $x_k$ be the operator in $\h_0$ whose matrix has $b_k$ in the first $k\times k$ corner and zero elsewhere. Note that $x_k$ is a rank one orthogonal projection and thus $\|x_k\|=1$. Also note that $\|P(x_k)\|=\frac1k$.
It follows that 
$$
\|x_k-P(x_k)\|\to 1
$$
as $k\to \infty$. On the other hand, $x_kA_0-A_0x_k$ is a Hilbert-Schmidt operator whose $2$-norm squared is 
\begin{eqnarray}
\|x_kA_0-A_0x_k\|_2^2 & = & \frac{1}{k^2} \sum_{i,j=1}^k \lambda_j^2+\lambda_i^2-2\lambda_j\lambda_i=\frac{2}{k^2}\{k\sum_{i=1}^k\lambda_i^2-(\sum_{i=1}^k \lambda_i)^2\}\nonumber\\
& \le& \frac2k \sum_{i=1}^k\lambda_i^2\le \frac2k\|A_0\|_2^2.\nonumber
\end{eqnarray}
Thus $\|x_kA_0-A_0x_k\|\le \|x_kA_0-A_0x_k\|_2\to 0$, leading to a contradiction. It follows that the spectrum of $A$ is finite.

The other implication is straightforward. We outline its proof. Suppose that the spectrum of $A$ is finite, then 
$A=\sum_{i=1}^n \lambda_i p_i$, for pairwise orthogonal self-adjoint projections $p_i$ which sum $1$. One can write operators in $\k(\h)$ as $n\times n$ matrices in terms of the decomposition $\h=\sum_{i=1}^n R(p_i)$. A straightforward computation shows that if $x\in\k(\h)$ with matrix $(x_{i,j})$, then $\delta_A(x)$ is
$$
\delta_A(x)= \left( \begin{array}{lllll}
0 & (\lambda_2-\lambda_1)x_{1,2} & (\lambda_3-\lambda_1)x_{1,3} & \dots & (\lambda_n-\lambda_1)x_{1,n} \\
(\lambda_1-\lambda_2)x_{2,1} & 0 & (\lambda_3-\lambda_2)x_{2,3} & \dots & (\lambda_n-\lambda_2)x_{2,n} \\
\dots & \dots & \dots & \dots & \dots \\
(\lambda_1-\lambda_n)x_{n,1} & (\lambda_2-\lambda_n)x_{n,2} & (\lambda_2-\lambda_1)x_{1,2} & \dots & 0
\end{array} \right).
$$
Since $\lambda_i\ne\lambda_j$ if $i\ne j$, it follows that 
the set $\{xA-Ax: x\in \k(\h)\}$ consists of operators in $\k(\h)$ whose $n\times n$ matrices have zeros on the diagonal, i.e.
$$
\{xA-Ax: x\in \k(\h)\}=\{z\in\k(\h): p_i zp_i=0 , \ i=1,\dots , n\},
$$
which is clearly closed in $\k(\h)$.
The kernel of $\delta_A$ consists of elements $x\in\k(\h)$ which are diagonal matrices, with an arbitrary compact operator of the range of $p_i$ in the $i,i$-entry. This space is clearly complemented, for instance, by the closed linear space of compact operators whose matrices have zeros in the diagonal.
\end{proof}

We may combine the above results to obtain the following:

\begin{teo}
The orbit $\o_A\subset A+\k(\h)_h$ is a differentiable complemented submanifold and  the map $\pi_A:U_c(\h)\to \o_A$ is a submersion if and only if the spectrum of $A$ is finite. 
\end{teo}
\begin{proof}
Suppose first that the orbit is a complemented submanifold and the map $\pi_A$ a submersion.  Then the differential 
$$
(d\pi_A)_1=\delta_A:(TU_c(\h))_1=\k(\h)_{ah}\to (T\o_a)_A\subset \k(\h)_h
$$
splits. In particular this implies that $(T\o_a)_A=\delta_A(\k(\h)_{ah})$ and therefore by the above theorem, $A$ has finite spectrum. 

Suppose that $A$ has finite spectrum, $A=\sum_{i=1}^n \lambda_i p_i$, with  $p_i$ pairwise orthogonal projections with $\sum_{i=1}^n p_i=1$. In order to apply Lemma \ref{lemaraeburn}, we need to show that $\pi_A$ is open (regarded as a map from $U_c(\h)$ onto $\o_A$). Let us show that it has continuous local cross section. The construction proceeds analogously as in \cite{upe}. Denote by $P_A$ the map $P_A(x)=\sum_{i=1}^np_i x p_i$. $P_A$ projects onto the commutant of $A$, and it has the following expectation property: if $y,z$ commute with $A$, then $P_A(yxz)=yP_A(x)z$. Moreover, by the argument above, there exists a constant $C>0$ such that 
$$
C\|x-P_A(x)\|\le \|xA-Ax\| ,
$$
for all $x\in \k(\h)$. Consider the open ball $B=\{b\in \o_A:\|b-A\|<C\}$. We define the following map on $B$:
$$
\theta:B\to U_c(\h) , \ \ \theta(b)= u \Omega(P_A(u^*)) , \ \ \hbox{ if } b=uAu^*,
$$
where $\Omega$ is the unitary part in the polar decomposition of an (invertible) operator, $g=\Omega(g)|g|$. Note that $\Omega(g)=g|g|^{-1}=g(g^*g)^{-1/2}$ is a smooth map of $g$, and takes values in $U_c(\h)$ if $g$ is an invertible element in the C$^*$-algebra $\mathbb{C}1+\k(\h)$. In order to verify that $\theta$ is well defined, let us check first that $P_A(u^*)$ is invertible and lies in $\mathbb{C}1+\k(\h)$. Since $b=uAu^*$, with $u=1+k$, $k\in\k(\h)$,
$$
C\|u-P_A(u)\|=C\|k-P_A(k)\|\le \|kA-Ak\|=\|uA-Au\|=\|uAu^*-A\|<C.
$$
Then $\|1-P_A(u)u^*\|=\|u-P_A(u)\|<1$, and thus $P_A(u)$ and $P_A(u)^*=P_A(u^*)$ are invertible. Also $P_A(u)-1=P_A(u-1)=P_A(k)\in\k(\h)$. Next note that $\theta(b)$ depends on $b$ and not on the choice of the unitary $u$. If also $b=wAw^*$, then $v=u^*w$ commutes with $A$. By uniqueness properties of the polar decomposition, $\Omega(v^*P_A(u^*))=v^*\Omega(P_A(u^*))$, and thus
$$
w \Omega(P_A(w^*))=uv \Omega(v^*P_A(w^*))=u \Omega(P_A(u^*)).
$$
Let us prove that $\theta$ is continuous. It suffices to show that it is continuous at $A$. Suppose that $u_nAu_n^*\to A$. Then  as above, $\|u_nA-Au_n\|\to 0$, and therefore $\|u_n-P_A(u_n)\| \to 0$, or equivalently,
$$
\|1-u_n P_A(u_n^*)\|=\|1-P_A(u_n)u_n^*\|=\|u_n-P_A(u_n)\|\to 0.
$$
Therefore (since $\Omega$ is continuous), $\theta(u_nAu_n^*)=u_n\Omega(P_A(u_n^*))=\Omega(u_n P_A(u_n^*))\to 1$.
Finally, $\theta$ is a cross section: if $b=uAu^*$,
$$
\theta(b)A\theta(b)^*=u\Omega(P_A(u^*)) A \Omega(P_A(u^*))^* u^*=uAu^*=b,
$$
because the fact that $P_A(u^*)$ commutes with $A$, implies that also $\Omega(P_A(u^*))$ commutes with $A$.
\end{proof}

\section{Finsler metric in the unitary orbit}

Our next step will be to introduce a Finsler metric in $\o_A$. To do this, we shall follow ideas in \cite{duranmatarecht}, and consider a quotient metric in the tangent spaces of $\o_A$ as follows. For any $b\in \o_A$, the map $\pi_b:U_c(\h)\to \o_A$, $\pi_b(u)=ubu^*$ is clearly a submersion (note that if $b=wAw^*$, then  $\pi_b(u)=\pi_A\circ ad(w)$). Therefore
$$
(d\pi_b)_1=\delta_b:\k(\h)_{ah}\to (T\o_a)_b=R(\delta_b)
$$ 
is surjective (here $\delta_b(y)=yb-by$). Thus if $x\in R(\delta_b)$,
\begin{equation}\label{lametrica}
\|x\|_b=\inf\{ \|y\|: y\in\k(\h)_{ah}  \hbox{ such that }  \delta_b(y)=x\}.
\end{equation}

Let us recall two relevant facts from \cite{duranmatarecht}:
\begin{enumerate}
\item
The metric is invariant under the action of the group: if $x\in R(\delta_b)$ and $u\in U_c(\h)$,
$$
\|x\|_b=\|uxu^*\|_{uau^*}.
$$
This is a straightforward verification.
\item
There exists $z\in\b(\h)_{ah}$ such that $\delta_b(z)=y$ which realizes the infimum: $\|z\|=\|x\|_b$. Such element $z$ is called a {\it minimal lifting} for $x$. We remark that it may not be unique, and more relevant to our particular situation, it may not be compact. Indeed, it is obtained as a weak limit point of a minimizing sequence, as follows. Let $y_n^*=-y_n$ with $\delta_b(y_n)=x$ such that $\|y_n\|\to \|x\|_b$. The sequence  $\{y_n\}$ is bounded, therefore there exists a subsequence, still denoted $y_n$, which is convergent in the weak operator topology, $\langle y_n \xi,\eta\rangle \to \langle z\xi,\eta\rangle $ for all $\xi,\eta\in\h$. Note that 
$$
x=\delta_b(y_n)=y_nb-by_n\stackrel{wot}{\to} zb-bz=\delta_b(z),
$$
and thus $\delta_b(z)=x$. Also it is clear that $z^*=-z$. For $\epsilon>0$, fix a unit vector $\xi$ such that $|\langle z\xi,\xi\rangle |\ge \|z\|-\epsilon$.
Then 
$$
\|y_n\|\ge |\langle y_n\xi,\xi \rangle |\to |\langle z\xi,\xi\rangle |\ge \|z\|-\epsilon,
$$
and therefore taking limits
$\|z\|\ge\|x\|_b \ge \|z\|-\epsilon$. It is not clear however if there is a compact minimal lifting. 
\end{enumerate}
The interest in minimal liftings, is that they provide curves of minimal length for the Finsler metric just defined in the homogeneous space $\o_A$. We shall prove here the existence of compact minimal liftings in a special situation, when the operator $A$ has finite rank. 

\begin{prop}
Suppose that $A$ has finite rank. Then for any $b\in\o_A$ and any tangent vector $x\in R(\delta_b)$ there exists a minimal lifting  $z\in \k(\h)_{ah}$, i.e. an element that verifies
$$
\delta_b(z)=x \  \hbox{ and } \ \|z\|=\|x\|_b=\inf\{ \|y\|: y\in\k(\h)_{ah}  \hbox{ such that }  \delta_b(y)=x\}.
$$
\end{prop}
\begin{proof}
We may suppose $b=A$. Since $A$ has finite rank, there exist pairwise orthogonal projections $p_0,p_1,\dots p_n$ such that $A=\sum_{i=1}^n\lambda_i p_i$, with $p_1,\dots,p_n$ of finite rank and $p_0=1-\sum_{i=1}^n p_i$, the projection onto the kernel of $A$.
Let $z\in\k(\h)_{ah}$ be an arbitrary lifting of $x$, $x=zA-Az$. As remarked above, there exists $z_0\in\b(\h)_{ah}$ such that $x=z_0A-Az_0$ and $\|z_0\|=\|x\|_A$. Consider the anti-hermitian element $z_1=z_0-p_0z_0p_0$. It is clearly also a lifting of $x$, and moreover, it is of finite rank. Indeed,  
$$
z_1=p_0^\perp z_0 +p_0z_0p_0^\perp,
$$
with $p_0^\perp=\sum_{i=1}^n p_i$ of finite rank. 

Recall Krein's solution \cite{krein} to the extension problem for a self-adjoint operator. Given an incomplete $2\times 2$ hermitian block operator matrix of the form
$$
\left( \begin{array}{ll} X & Y \\ Y^* & \ \end{array} \right),
$$
find a completion (i.e. a self-adjoint operator  $Z$ completing the $2,2$ entry) in order that the complete matrix has minimal norm. Krein proved that there is always a solution, and that it may not be unique. More recently, Davis, Kahane and Weinberger \cite{davis} gave explicit formulas for $Z$, and in particular, in \cite[Corollary 1.3]{davis} proved that if the initial operator matrix is compact, then one can always find a compact solution $Z$.

In our case the operators are anti-hermitian, but clearly the result can be translated.  The compact operator $z_1$, when regarded as a $2\times 2$ matrix in terms of the decomposition $\h=R(p_0^\perp)\oplus R(p_0)$, has a zero in the $2,2$ entry. Thus it can be completed to a compact operator with minimal norm. That is, there exists a compact operator $c=p_0cp_0$ such that $z_c:=z_1+c$ verifies
$$
\|z_c\|\le \|z_1+p_0dp_0\|
$$
for any $d\in\b(\h)_{ah}$. In particular, $\|z_c\|\le \|z_0\|$. Clearly $\delta_A(z_c)=\delta_A(z_1)=\delta_A(z_0)=x$,
therefore in fact $\|z_c\|=\|z_0\|$. Also it is apparent that $z_c$ is compact.
\end{proof}
From minimal liftings one obtains curves with minimal length. First recall that if $\gamma(t)$, $t\in[0,1]$, is a smooth curve in $\o_A$, one measures its length as usual
$$
L(\gamma)=\int_0^1 \|\dot{\gamma}(t)\|_{\gamma(t)} d t.
$$
If $b\in \o_A$ and $x\in R(\delta_b)$, then for any lifting $z\in\k(\h)_{ah}$, the curve $\gamma(t)=e^{tz}be^{-tz}$ verifies $\gamma(0)=b$ and $\dot{\gamma}(0)=x$. If additionally $z=z_c$ is a minimal lifting, the curve $\gamma$ has minimal length up to a critical value of $t$. This could be proved using the techniques developed by Dur\'an, Mata-Lorenzo and Recht in \cite{duranmatarecht}. We give here a different  proof based on the convexity of the distance function proved in the first section. 
First we need the following result:
\begin{prop}\label{rectifiable}
Let $a_0$, $a_1$ in $\o_A$, with $A$ of finite rank, and let $d$ be the rectifiable distance in $\o_A$. Then
$$
d(a_0,a_1)=\inf\{L(\Gamma): \Gamma(t)\in U_c(\h), \pi_{a_0}(\Gamma) \hbox{ joins } a_0 \hbox{ and } a_1\}.
$$
The curves $\Gamma$ considered are continuous and piecewise smooth.
\end{prop}
\begin{proof}
Let $\gamma(t)\in\o_A$ be a $C^1$ curve joining $\gamma(0)=a_0$ and $\gamma(1)=a_1$. Since $\pi_a:U_c(\h)\to\o_A$ are submersions for all $a\in\o_A$, there exists a continuous piecewise smooth curve $\Gamma$ in $U_c(\h)$ such that $\pi_{a_0}(\Gamma(t))=\gamma(t)$, $t\in[0,1]$. Differentiating this equality, one obtains $\dot{\gamma}=(d\pi_{a_0})_{\Gamma}(\dot{\Gamma})$.
Elementary calculations show that $(d\pi_a)_u= \delta_{uau^*}\circ r_{u^*}$, where $r_{u^*}(x)=xu^*$.
It follows that 
$$
\dot{\delta}=\delta_{\gamma}(\dot{\Gamma}\Gamma^*).
$$
Note that $\dot{\Gamma}\Gamma^*$ belongs to $\k_{ah}(\h)$, and therefore 
$$
\|\dot{\gamma}\|_\gamma\le \|\dot{\Gamma}\Gamma^*\|=\|\dot{\Gamma}\|.
$$
Thus $L(\gamma)\le L(\Gamma)$, and $d(a_0,a_1)\le L(\Gamma)$. To finish, we must prove that one can arbitrarily approximate $L(\gamma)$ with lengths of curves in $U_c(\h)$ joining the fibers of $a_0$ and $a_1$.
Fix $\epsilon>0$. Let $0=t_0<t_1<\dots t_n=1$ be a uniform partition of $[0,1]$ ($\Delta t_i=t_i-t_{i-1}=1/n$) such that the following hold:
\begin{enumerate}
\item
$\|\dot{\gamma}(s)-\dot{\gamma}(s')\|< \epsilon/4$ if $s,s'$ lie in the same interval $[t_{i-1},t_i]$.
\item
$\left|L(\gamma)-\sum_{i=0}^{n-1} \|\dot{\gamma}(t_i)\|_{\gamma(t_i)} \Delta t_i\right|<\epsilon/2$.
\end{enumerate}
For each $i=0,\dots ,n-1$, let $z_i\in\k_{ah}(\h)$ be a minimal lifting for $\dot{\gamma}(t_i)$, i.e.
$$
\delta_{\gamma(t_i)}(z_i)=\dot{\gamma}(t_i) \ \hbox{ and } \ \ \|\dot{\gamma}(t_i)\|_{\gamma(t_i)}=\|z_i\|.
$$
Consider the following curve $\Omega$ in $U_c(\h)$:
$$
\Omega(t)=
\left\{
\begin{array}{ll}
e^{tz_0} & t\in [0,t_1) \\
e^{(t-t_1)z_1}e^{t_1z_0} & t\in [t_1,t_2) \\
e^{(t-t_2)z_2}e^{(t_2-t_1)z_1}e^{t_0z_0} & t\in [t_2,t_3) \\
\dots & \dots \\
e^{(t-t_{n-1})z_{n-1}}\dots e^{(t_2-t_1)z_1}e^{t_0z_0} & t\in [t_{n-1},1] 
\end{array}
\right.
$$
Clearly $\Omega$ is continuous and piecewise smooth, $\Omega(0)=1$ and
$$
L(\Omega)=\sum_{i=1}^{n-1}\|z_i\|\Delta t_i=\sum_{i=0}^{n-1} \|\dot{\gamma}(t_i)\|_{\gamma(t_i)} \Delta t_i.
$$
Let us show that $\pi_{a_0}(\Omega(1))$ lies close to $a_1$.
Indeed, first note that if we denote by $\alpha(t)=\pi_{a_0}(e^{tz_0})-\gamma(t)$, then $\alpha(0)=0$ and, using the mean value theorem in Banach spaces \cite{dieudonne},
$$
\|\pi_{a_0}(e^{t_1z_0})-\gamma(t_1)\|=\|\alpha(t_1)-\alpha(0)\|\le \|\dot{\alpha}(s_1)\|\Delta t_1,
$$
for some $s_1\in [0,t_1]$. Explicitly,
$$
\|\pi_{a_0}(e^{t_1z_0})-\gamma(t_1)\|\le \|e^{s_1z_0}\delta_{a_0}(z_0)e^{-s_1z_0}-\dot{\gamma}(s_1)\|\Delta t_1.
$$
Note that $\delta_{a_0}(z_0)=\dot{\gamma}(0)$, and that
$$
\|e^{s_1z_0}\dot{\gamma}(0)e^{-s_1z_0}-\dot{\gamma}(s_1)\|\le \|e^{s_1z_0}\dot{\gamma}(0)e^{-s_1z_0}-\dot{\gamma}(0)\|+\|\dot{\gamma}(0)-\dot{\gamma}(s_1)\|.
$$
The second summand is bounded by $\epsilon/4$. The first summand can be bounded as follows
\begin{eqnarray}
\|e^{s_1z_0}\dot{\gamma}(0)e^{-s_1z_0}-\dot{\gamma}(0)\|&=&\|z_0(1-e^{s_1z_0})a_0-a_0(1-e^{-s_1z_0})z_0\|\nonumber\\
&\le& 2\|z_0\|\|a_0\|\|1-e^{s_1z_0}\|\le D\Delta t_1,\nonumber
\end{eqnarray}
where $D$ is a constant which depends on $\gamma$. For instance, 
$$
\|z_i\|=\|\delta_{\gamma(t_i)}(\dot{\gamma}(t_i)\|_{\gamma(t_i)}\le \|\dot{\Gamma}(t_i)\|\le \max_{t\in[0,1]}\|\dot{\Gamma}(t)\|=D.
$$

It follows that 
$$
\|\pi_{a_0}(e^{t_1z_0})-\gamma(t_1)\|\le (D\Delta t_1+\epsilon/4)\Delta t_1.
$$
Next estimate $\|\pi_{a_0}(e^{(t_2-t_1)z_1}e^{t_1z_0})-\gamma(t_2)\|$, which by the triangle inequality is less or equal than
$$
\|e^{(t_2-t_1)z_1}e^{t_1z_0}a_0e^{-t_1z_0}e^{-(t_2-t_1)z_1}-e^{(t_2-t_1)z_1}\gamma(t_1)e^{-(t_2-t_1)z_1}\|
$$
$$
+\|e^{(t_2-t_1)z_1}\gamma(t_1)e^{-(t_2-t_1)z_1}-\gamma(t_2)\|.
$$
The first summand is
\begin{eqnarray}
\|e^{(t_2-t_1)z_1}e^{t_1z_0}a_0e^{-t_1z_0}e^{-(t_2-t_1)z_1}-e^{(t_2-t_1)z_1}\gamma(t_1)e^{-(t_2-t_1)z_1}\|=&&\nonumber\\
=\|e^{(t_2-t_1)z_1}(e^{t_1z_0}a_0e^{-t_1z_0}-\gamma(t_1))e^{-(t_2-t_1)z_1}\|=\nonumber&&\\
=\|(e^{t_1z_0}a_0e^{-t_1z_0}-\gamma(t_1))\|\le (D\Delta t_1+\epsilon/4)\Delta t_1.&&\nonumber
\end{eqnarray}
The second can be treated analogously as the first difference above,
$$
\|e^{(t_2-t_1)z_1}\gamma(t_1)e^{-(t_2-t_1)z_1}-\gamma(t_2)\|\le (D \Delta t_2+\epsilon/4)\Delta t_2.
$$
Thus (using that $\Delta t_i=1/n$) 
$$
\|\pi_{a_0}(e^{(t_2-t_1)z_1}e^{t_1z_0})-\gamma(t_2)\|\le (D/n+\epsilon/4)2/n.
$$
Inductively, one obtains that 
$$
\|\Omega(1)a_0 \Omega(1)^*-\gamma(1)\|\le D/n+\epsilon/4<\epsilon/2,
$$
choosing $n$ appropriately. The proof follows.
\end{proof}
\begin{teo}
Let $A=A^*$ be of finite rank, $b\in\o_A$ and $x\in R(\delta_b)$ a tangent vector with $\|x\|_b<\pi/2$. If $z_c$ is a (compact) minimal lifting of $x$, then the curve $\delta(t)=e^{tz_c}be^{-tz_c}$ has minimal length up to $|t|\le 1$.
\end{teo}
\begin{proof}
We may suppose $b=A$. By the previous proposition, in order to establish our result, it suffices to compare  the lengths of $\Delta(t)=e^{tz_c}$ and $\Gamma$, where $\Gamma$ is a piecewise smooth curve in $U_c(\h)$ joining $1$ and a unitary in the fiber of $\delta(1)$. Indeed, note  that $\Delta(t)$, which lifts $\delta$,  clearly  has the same length as $\delta$.  If $L(\Gamma)\ge \pi/2$, then it is longer than $\Delta$ (whose length is strictly less than $\|z_c\|<\pi/2$).  Otherwise,  $\Gamma(1)=e^y$, with $y\in\k(\h)_{ah}$ and $\|y\|<\pi/2$. Note that $\Gamma$ and $\Delta$ may have different endpoints, however both  $\Gamma(1)=e^y$ and $\Delta(1)=e^{z_c}$ verify
$$
e^yAe^{-y}=e^{z_c}Ae^{-z_c}
$$
and therefore $e^y=e^{z_c}e^{d}$ with $d\in\k(\h)_{ah}$ commuting with $A$. Moreover, the assumption on the lengths of $\Delta$ and $\Gamma$ implies that $\|d\|\le \pi$. Thus $\beta(t)=e^{z_c}e^{td}$, $t\in[0,1]$ is the minimal geodesic joining $e^{z_c}$ and  $e^y$. Consider the map 
$$
f(t)=d_\infty(1,\beta(t))=\|\log(e^{z_c}e^{td})\| \ , \ \ t\in[-1,1]
$$
of the previous section. We claim that it has a minimum on $t=0$. Since we know that $f$  is convex, it suffices to analyze the lateral derivatives at this point. By the Baker-Campbell-Hausdorff formula, the linear approximation of $\log(e^{z_c}e^{td})$ is given by
$$
\log(e^{z_c}e^{td})= z_c+td + R_2(z_c,td),
$$
where  $\lim_{t\to 0} \frac{\|R_2(z_c,td)\|}{t}=0$.
Consider first $t>0$, then
$$
\|z_c+td\|-\|R_2(z_c,td)\|\le \|\log(e^{z_c}e^{td})\|\le \|z_c+td\|+\|R_2(z_c,td)\|,
$$
and therefore
\begin{eqnarray}
\frac1t\{\|z_c+td\|-\|z_c\|\}-\frac1t\|R_2(z_c,td)\|&\le&  \frac1t\{\|\log(e^{z_c}e^{td})\|-\|z_c\|\}\nonumber\\
&\le& \frac1t\{\|z_c+td\|-\|z_c\|\}+\frac1t\|R_2(z_c,td)\|.\nonumber
\end{eqnarray}
If we take limit $t\to 0^+$, we obtain that the right derivative $\lim_{t\to 0^+}\frac1t\{\|z_c+td\|-\|z_c\|\}$ of the norm at $z_c$, $d$ (which exists due to the convexity of the norm, see for instance \cite{phelps}), coincides with the right derivative $\partial^+ f(0)$ of $f$ at $t=0$. Indeed, note in the middle term of the above inequalities, that $\|z_c\|=f(0)$. Since $z_c$ is a minimal lifting, and $d\in\k(\h)_{ah}$ commutes with $A$, it follows that $\|z_c+td\|\ge \|z_c\|$, i.e. $\partial^+f(0)\ge 0$. Analogously one proves that $\partial^- f(0)\le 0$. Thus $f(0)\le f(t)$ for all $t\in[-1,1]$. In particular, 
$$
L(\Gamma)\ge\|y\|=f(1)\ge f(0)=\|z_c\|=L(\Delta).
$$
\end{proof}
\begin{rem}
In fact, with the same argument as above, it can be proved that if $A$ has finite spectrum, $A=\sum_{i=1}^n\lambda_i p_i$, and only one of the $p_i$ has infinite rank, then there exist compact minimal liftings for any vector tangent to the orbit of $A$. 

This condition is not necessary, as is shown in the example of the following section.
\end{rem}  

\section{The orbit of an infinite projection}

In the special case when $A=P$ is a projection more is known on the geometry of the unitary orbit. For instance, for arbitrary C$^*$-algebras \cite{brown,pr},  two projections at norm distance strictly less than $1$ can be joined by a minimal geodesic. What is not known in the general case is whether two projections $p_1,p_2$ which are unitarily equivalent and verify $\|p_1-p_2\|=1$, can be joined by a minimal geodesic. 

The case in study in this paper corresponds to consider the C$^*$-algebra  ${\mathbb C} 1+\k(\h)$, the unitization of the algebra of compact operators. Indeed, the unitary group $U_{{\mathbb C} 1+\k(\h)}$ of this algebra consists of unitary operators of the form $u=\lambda 1+c$ with $c$ compact and $\lambda\in{\mathbb C}$ with $|\lambda|=1$. Clearly $U_c(\h) \subset U_{{\mathbb C} 1+\k(\h)}$. However, clearly the orbits coincide, 
$$
\o_P=\{uPu^*:u\in U_c(\h)\}=\{uPu^*: u\in U_{{\mathbb C} 1+\k(\h)}\}.
$$
Let us prove that in $\o_P$, any pair of elements can be joined by a minimal geodesic. Given a self-adjoint projection $p$, an operator $x$ is $p$ co-diagonal if its $2\times 2$ matrix in terms of $p$ is co-diagonal, i.e. $pxp=(1-p)x(1-p)=0$. That $p$ co-diagonal elements provide  minimal geodesics, was proved in \cite{pr}: namely if $x\in\k(\h)$ is $p$ co-diagonal, with $\|x\|\le\pi/2$ then the curve $\gamma(t)=e^{tx}pe^{-tx}$ is minimal for $t\in[0,1]$.

Before we establish our result, we must recall certain facts on the orbit of an infinite projection by the action of the Fredholm unitary group. In \cite{stratilavoiculescu} (see also \cite{carey,odospe}), it was proved that this orbit fills the connected component of $p$ in the so called Sato Grassmannian of the decomposition $\h=R(p)\oplus\ker p$ (also called restricted Grassmannian of the decomposition \cite{pressleysegal}). It consists of all projections $q$ such that 
$$
q|_{R(p)}:R(p)\to R(q) \ \ \hbox{ is a Fredholm operator and }
$$
$$
q|_{\ker p}:\ker p\to R(q) \ \ \hbox{ is compact}.
$$

\begin{teo}
Let $p_0,p_1\in \o_p$. Then there exists a minimal curve joining them. In other words, there exists $z\in\k(\h)_{ah}$ , $\|z\|\le \pi/2$, and $z$ is  $p_0$ co-diagonal, such that
$$
p_1=e^zp_0e^{-z}.
$$
\end{teo}
\begin{proof}
If $\|p_0-p_1\|<1$, in \cite{pr} it was proved that there  exists a unique  $z$, $p_0$ co-diagonal, which implements the geodesic. It is explicitly computed in terms of $p_0$ and $p_1$: namely consider the symmetries (self-adjoint unitary operators)
$\epsilon_i=2p_i-1$, $i=0,1$, which verify $\|\epsilon_0-\epsilon_1\|<2$. Thus $\|1-\epsilon_1\epsilon_0\|<2$, and $z$ is 
$$
z=\frac12 \log (\epsilon_1\epsilon_0)\in \b(\h)_{ah}.
$$
Since $p_i\in\o_p$, it follows that $\epsilon_1\epsilon_0\in 1+\k(\h)$, and thus $z\in\k(H)_{ah}$.

Clearly it suffices to consider the case $\|p_0-p_1\|=1$. 

Consider the following subspaces:
$$
\h_{00}=\ker p_0\cap \ker p_1 \ , \ \ \h_{01}=\ker p_0\cap R(p_1) \ ,
$$
$$
 \ \ \h_{10}=R(p_0)\cap \ker p_1\ , \ \ \h_{11}=R(p_0)\cap R(p_1) \ ,
$$
and
$$
\h_0=(\h_{00}\oplus \h_{01} \oplus \h_{10} \oplus \h_{11})^\perp.
$$
These are the usual subspaces to regard when considering the unitary equivalence of
two projections \cite{dixmier}. The space $\h_0$ is usually called the generic part
of $p_0$ and $p_1$. It is invariant both for $p_0$ and $p_1$. Also it is
clear that $\h_{00}$ and $\h_{11}$ are invariant for $p_0$ and $p_1$, and that $p_0$ and
$p_1$ coincide here. Thus in order to find a unitary operator $e^z$ conjugating $p_0$
and $p_1$, with $z\in \k(\h)_{ah}$,  which is $p_0$ co-diagonal, and such that
$\|z\|\le \pi/2$, one needs to focus on the subspaces $\h_0$ and $\h_{01}\oplus \h_{10}$.

Let us treat first $\h_0$, denote by $p_0'$ and $p'_1$ the projections $p_0$ and $p_1$ reduced to $\h_0$. These projections are in generic position \cite{dixmier,halmos}. In \cite{halmos} Halmos showed that two projections in generic position are unitarily equivalent, more specifically, he showed that there exists a unitary operator $w:\h_0\to \jj\times \jj$ such that 
$$
wp_0'w*=p_0''=\left( \begin{array}{cc} 1 & 0 \\ 0 & 0 \end{array} \right) \ \ \hbox{ and } \ \ 
wp'_1w*=p''_1=\left( \begin{array}{cc} c^2 & cs \\ cs & s^2 \end{array} \right),
$$
where $c,s$ are positive commuting contractions acting in $\jj$ and satisfying $c^2+s^2=1$. We claim that there exists an anti-hermitian operator $y$ acting on $\jj\times \jj$, which is a co-diagonal matrix, and  such that $e^yp_0''e^{-y}=p''_1$. In that case, the element $z_0=w^*yw$ is an anti-hermitian operator in $\h_0$, which verifies $e^{z_0}p_0'e^{-z_0}=p'_1$, and is co-diagonal with respect to $p_0'$. Moreover, we claim that $y$ is a compact operator in $\jj\times \jj$ with $\|y\|\le \pi/2$, so that $z_0$ is also a compact operator in $\h_0$ with $\|z_0\|\le\pi/2$. Let us prove these claims. By a functional calculus argument, there exists a positive element $x$ in the C$^*$-algebra generated by $c$, with $\|x\|\le \pi/2$,  such that $c=\cos(x)$ and $s=\sin(x)$. Since $p''_1$ lies in the restricted Grassmannian of $p_0''$, in particular one has that $p''_1|_{\ker(p_0'')}$ is a compact operator. That is, the operator $\cos(x)\sin(x)+\sin^2(x)$ is compact in $\jj$. By a straightforward functional calculus argument, it follows that $x$ is a compact operator. Consider the operator
$$
y=\left( \begin{array}{cc} 0 & -x \\ x & 0 \end{array} \right)
$$
Clearly $y^*=-y$, $\|y\|\le \pi/2$. A straightforward computation shows that 
$$
e^yp_0''e^{-y}=p''_1,
$$
and our claims follow.

Let us consider now the space  $\h_{01}\oplus \h_{10}$. Recall
\cite{pressleysegal} that with an equivalent definition,  if $p_1$ lies in the connected component of $p_0$ (in the restricted Grassmannian) then 
$$
p_0p_1|_{R(p_1)}:R(p_1)\to R(p_0)
$$
is a Fredholm operator of index $0$. Note that $\h_{01}=\ker(p_0p_1|_{R(p_1)})$. Thus
in particular $\dim \h_{01}<\infty$.
On the other hand, it is also apparent that $\h_{10}\subset R(p_0p_1)^\perp \cap R(p_0)$,
and therefore also $\dim \h_{10}<\infty$. Therefore, the fact that $p_0p_1|_{R(p_1)}$
has zero index implies that 
$$
\dim \h_{01}\le \dim \h_{10}.
$$
The fact that $p_1$   lies in the connected component of $p_0$ in the restricted Grassmannian
corresponding to the polarization given by $p_0$, implies that, reciprocally, $p_0$ lies
in the component of $p_1$, in the Grassmannian corresponding to the polarization
given by $p_1$. Thus, by symmetry, 
$$
\dim \h_{01}= \dim \h_{10}.
$$
Let $v:\h_{10}\to \h_{01}$ be a surjective isometry, and consider 
$$
w: \h_{01}\oplus \h_{10}\to \h_{01}\oplus \h_{10} \ , w(\xi'+\xi'')=v^*\xi'+v\xi''.
$$
In matrix form (in terms of the decomposition $\h_{01}\oplus \h_{10}$), 
$$
w=
\left(
\begin{array}{ll}
0 & v \\ v^* & 0
\end{array}
\right).
$$
Apparently, $w p_0|_{\h_{01}\oplus \h_{10}} w^*=p_1|_{\h_{01}\oplus \h_{10}}$. Let 
$$
z_{2}= \pi/2
\left(
\begin{array}{cc}
0 & v \\ -v^* & 0
\end{array}
\right).
$$
Note that $z_2$ is an anti-hermitian operator in $\h_{01}\oplus \h_{10}$, with norm
equal to $\pi/2$.
 A straightforward matrix computation shows that $e^{z_2}=w$. Consider now
 $$
 z=z_0+z_1+z_2,
 $$
 where $z_1=0$ in  $\h_{00}\oplus \h_{11}$, and $z_0$ is the anti-hermitian operator in
the generic part $\h_0$ of $\h$ found above. Then it is clear that $z$ is
anti-hermitian, $z$ is compact ($\dim(\h_{01}\oplus \h_{10})<\infty$), $z$ is $p_0$ co-diagonal, $\|z\|=\pi/2$, and  $e^zp_0e^{-z}=p_1$.

\end{proof}
\section{The unitary orbit of a non self-adjoint operator}
In this section we consider the orbit of the nilpotent operator $N\in \b(\h)$, $\h=\h_0\times \h_0$,
$$
N=\left( \begin{array}{ll} 0 & I_{\h_0} \\ 0& 0\end{array} \right).
$$
The isotropy group of the action of $U_c(\h)$ consists of unitary operators of this group which are of the form
$$
v=\left( \begin{array}{ll} v_0 & 0 \\ 0& v_0 \end{array} \right).
$$
Note that $v_0-I_{\h_0}$ is compact in $\h_0$. Let 
$$
\pi_N:U_c(\h)\to \o_N=\{uNu^*: u\in U_c(\h)\} \ , \ \ \pi_N(u)=uNu^*,
$$
and $\delta_N$ its differential at $I$,
$$
\delta_N:\k(\h)_{ah}\to \k(\h) \ ,  \ \ \delta_N(x)=xN-Nx.
$$
It is apparent that the kernel of $\delta_N$, which consists of matrices of the form
\begin{equation}\label{kerdedeltaN}
y=\left( \begin{array}{ll} y_0 & 0 \\ 0& y_0 \end{array} \right), 
\end{equation}
with $y_0\in\k(\h_0)_{ah}$, is complemented in $\k(\h)_{ah}$. The range of $\delta_N$ is given by
$$
R(\delta_N)=\{ \left( \begin{array}{cc} a & b \\ 0 & -a \end{array} \right) : a\in\k(\h_0), b\in \k(\h_0)_{ah}\}.
$$
This space is clearly complemented in $\k(\h)$. For instance, a supplement is
$$
\mathbb{S}=\{\left( \begin{array}{ll} a' & b' \\ c' & a' \end{array} \right): a',c'\in\k(\h_0), b'\in\k(\h_0)_h \}.
$$
We want to use again Lemma \ref{lemaraeburn}, to prove that $\o_N$ is a differentiable submanifold of $\k(\h)$. To do this, it only remains to prove that $\pi_N$ is open.
\begin{prop}
$\pi_N:U_c(\h) \to \o_N$ has continuous local cross sections.
\end{prop}
\begin{proof}
If $b\in\o_N$ consider $s(b)=bb^*NN^*+b^*N$. Note that $s:\o_N\to \b(\h)$ is continuous and that $s(N)=NN^*+N^*N=1$,
therefore the set $\u_N=\{b\in\o_N: s(b) \hbox{ is invertible}\}$ is open in $\o_N$. Note  that since $b^2=0$ and $bb^*b=b$, one has that
$$
bs(b)=bb^*N=s(b)N,
$$
and
$$
s(b)^*s(b)N=NN^*bb^*N=Ns(b)^*s(b).
$$
The second identity implies that the absolute value $|s(b)|=(s(b)^*s(b))^{1/2}$ commutes with $N$. If $b\in \u_N$, the first identity implies that if $\mu(b)$ equals the unitary part in the polar decomposition $s(b)=\mu(b)|s(b)|$, then $\mu(b)N\mu(b)^*=b$. Thus one obtains a local cross section for $\pi_N$ on the open neighborhood $\u_N$ of $N$ in $\o_N$. Moreover, it takes values in $U_c(\h)$. Indeed, if $b=uNu^*$ for some $u\in U_c(\h)$ with $u-1=k$, then 
\begin{eqnarray}
s(b)&=&uNN^*u^*NN^*+uN^*u^*N\nonumber\\
&=&kNN^*(k^*NN^*+1)+NN^*k^*NN^*+(kN^*+N^*)k^*N+\nonumber\\
&& +kN^*N+NN^*+N^*N \in 1+\k(\h),\nonumber
\end{eqnarray}
because $NN^*+N^*N=1$. Therefore $\mu(b)$ is a unitary element in the C$^*$-algebra $\mathbb{C}1+\k(\h)$. Note that $s(b)\in 1+\k(\h)$, which implies that $|s(b)|\in 1+\k(\h)$, and thus in fact $\mu(b)\in U_c(\h)$. 

Cross sections on neighborhoods around other points of $\o_N$ are obtained by translation with the group action.
\end{proof}

\begin{coro}
The unitary orbit $\o_N\subset N+\k(\h)$ is a real analytic submanifold, and the map $\pi_N:U_c(\h)\to \o_N$ is a real analytic submersion.
\end{coro}
Consider the following Finsler metric, which is analogous to the metric in the orbit of a self-adjoint operator:
 if $x=T(\o_N)_b=\delta_N(\k(\h))$, then
 $$
 \|x\|_b=\inf\{\|z\|: z\in\k(\h)_{ah} \hbox{ with } \delta_b(z)=x\},
 $$
 where $\delta_b(a)=ab-ba$ as usual.
Let us show now that, for certain tangent vectors, which we shall call  anti-symmetric, one can find minimal geodesics of this Finsler metric, having these symmetric vectors as initial velocity. A general tangent vector at $N$ is an operator of the form $x=\delta_N(z)$ with $z^*=-z$. In matrix form, it is 
 $$
 x=\left( \begin{array}{ll} x_0 & x_1 \\ 0 & x_0 \end{array}\right),
 $$
with $x_0,x_1\in \k(\h_0)$, $x_1^*=-x_1$. We shall say that $x$ is {\it anti-symmetric} if $x_0$ is also anti-hermitian. Equivalently, this conditions means that $x$ has liftings $z$ of the form
 $$
 z=\left( \begin{array}{ll} z_{11} & z_{12} \\ -z_{12} & z_{22} \end{array}\right),
 $$  
 with all entries anti-hermitian. A vector $x$ tangent at $b=uNu^*$ is called anti-symmetric, if $u^*xu$ (which is tangent at $N$) is anti-symmetric in the above sense. Note that this does not depend on the choice of $u$.

\medskip
 
Let us show that anti-symmetric tangent vectors have  compact minimal liftings, and that they can be explicitly computed. It suffices to show this fact at $N$.
 \begin{lem}
Let $x=\delta_N(z)$ anti-hermitian, $z\in\k(\h)_{ah}$ with $z_{12}^*=-z_{12}$, then there exists $z_0\in\k(\h)_{ah}$ a minimal lifting of $x$. Namely,
 if 
 $$
 z=\left( \begin{array}{ll} z_{11} & z_{12} \\ -z_{12} & z_{22} \end{array} \right),
 $$
 then a minimal (compact) lifting is given by
 $$
 z_0=\left( \begin{array}{ll} \frac12(z_{11}-z_{22}) & z_{12} \\ -z_{12} & \frac12(z_{22}-z_{11}) \end{array} \right).
 $$
 \end{lem}
\begin{proof}
The operators $y, z$ here are anti-hermitian, thus they are of the form  $y=iy'$, $z=iz'$, with $y',z'$ hermitian.  In order to lighten the notation we shall reason with hermitian operators, and denote them by $y,z$.
Denote by $\Delta=\frac12(z_{11}-z_{22})$. In order to prove that $z_0$ is a minimal lifting, one has to show that
$$
\|z_0+y\|\ge \|z_0\|,
$$
for all $y=y*\in\ker \delta_N$.  Let $\xi=(\xi_1,\xi_2)\in \h_0\times \h_0$ with $\|\xi\|=1$.
Then 
$$
\langle z_0\xi,\xi\rangle =\langle \Delta \xi_1,\xi_1\rangle -\langle \Delta \xi_2,\xi_2\rangle +2 Re \langle z_{12}\xi_2,\xi_1\rangle.
$$
Note that if $\eta=(-\xi_2,\xi_1)$, then $\|\eta\|=1$, and 
$$
\langle z_0\eta,\eta\rangle =-\langle z_0\xi,\xi\rangle.
$$
The key fact here is that $z_{12}^*=z_{12}$.
It follows that both $-\|z_0\|$ and $+\|z_0\|$ belong to the spectrum of $z_0$. Let $\xi^n=(\xi^n_1,\xi^n_2)\in\h_0\times \h_0$ with $\|\xi^n\|=1$, such that $\langle z_0\xi^n,\xi^n\rangle\to \|z_0\|$. Taking $\eta^n=(-\xi^n_2,\xi^n_1)$ as above, one has that $\langle z_0\eta^n,\eta^n\rangle \to -\|z_0\|$.
If $y\in \ker \delta_N$, by equation (\ref{kerdedeltaN}) above,
$$
\langle (z_0+y)\xi^n, \xi^n\rangle =\langle z_0\xi^n,\xi^n\rangle + \langle y_0 \xi^n_1,\xi^n_1 \rangle +\langle y_0 \xi^n_2,\xi^n_2\rangle.
$$
Also note that $\langle y\eta^n,\eta^n\rangle=\langle y\xi^n,\xi^n\rangle=r_n$, which is a bounded sequence in $\mathbb{R}$. Consider a convergent subsequence of these numbers, and denote it again by $r_n$, with $r_n\to r_0$. 
Then 
$$
\|z_0+y\|\ge \langle(z_0+y)\xi^n, \xi^n\rangle\to \|z_0\|+r_0,
$$
and 
$$
-\|z_0+y\|\le \langle (z_0+y)\eta^n, \eta^n\rangle \to -\|z_0\|+r_0,
$$
Therefore if either $r_0\ge 0$ or $r_0<0$, one has that $\|z_0+y\|\ge \|z_0\|$.
It is apparent that $z_0$ is compact.
\end{proof}

\begin{teo}
Let $b\in\o_N$ and $x\in(T\o_N)_b$  an anti-symmetric tangent vector with $\|x\|_b<\pi/2$. Then there exists a curve of the form
$\gamma(t)=e^{tz_0}be^{-tz_0}$ in $\o_N$, with $z_0$ a minimal lifting of $x$, which verifies
$$
\gamma(0)=b , \ \ \dot{\gamma}(0)=x
$$
and such that $\gamma$ is minimal for $|t|\le 1$.
\end{teo}
\begin{proof}
The proof proceeds as in the self-adjoint case. One reduces to $b=N$, and uses the analogous of Proposition \ref{rectifiable}, which is proved similarly. The result thus rests on the local convexity property of the geodesic distance of $U_c(\h)$.
\end{proof}

\bigskip

\noindent
Esteban Andruchow and Gabriel Larotonda\\
Instituto de Ciencias \\
Universidad Nacional de General Sarmiento \\
J. M. Gutierrez 1150 \\
(B1613GSX) Los Polvorines \\
Buenos Aires, Argentina  \\
e-mails: eandruch@ungs.edu.ar, glaroton@ungs.edu.ar

\end{document}